\documentclass[11pt]{article}
\usepackage{amsmath,amsthm,amsfonts,amssymb,amscd, amsxtra, mathrsfs}
\usepackage{url}
\usepackage[margin=2.7 cm,nohead]{geometry}
\usepackage{color}
\usepackage{graphicx,color,rotating}
\usepackage{amsmath,amsthm}
\usepackage{amssymb}

\theoremstyle{plain}
\newtheorem{theorem}{Theorem}[section]
\newtheorem{lemma}[theorem]{Lemma}

\newtheorem{proposition}[theorem]{Proposition}

\theoremstyle{definition}
\newtheorem{definition}[theorem]{Definition}
\newtheorem{example}[theorem]{Example}

\theoremstyle{remark}
\newtheorem{remark}{Remark}

\newtheorem{algorithm}{Algorithm}

\newcommand{\Minimize}{\rm Minimize}

\newcommand{\argmin}{\rm argmin}

\newcommand{\R}{\mathbb R}

\def \T {{\scriptscriptstyle\mathrm{T}}} 

\def \T {{\scriptscriptstyle\mathrm{T}}} 
\def\R{\mathbb R}
\begin{document}
\title{On the Frank-Wolfe algorithm  for  non-compact constrained optimization  problems    \thanks{The first author  was supported in part by  CNPq grants 305158/2014-7 and 302473/2017-3,  FAPEG/PRONEM- 201710267000532 and CAPES. The second author was supported in part by Funda\c{c}\~{a}o de Apoio \`{a} Pesquisa do Distrito Federal (FAP-DF) by the grant 0193.001695/2017 and PDE 05/2018}}
\author{
O.  P. Ferreira\thanks{Instituto de Matem\'atica e Estat\'istica, Universidade Federal de Goi\'as,  CEP 74001-970 - Goi\^ania-GO, Brazil, E-mail:{\tt  orizon@ufg.br}.}
\and
W. S. Sosa\thanks{Programa de P\'os-Gradua\c c\~ao em Economia, Universidade Cat\'olica de Bras\'ilia,  CEP 70790-160 - Bras\'ilia-DF, Brazil, E-mail: {\tt  sosa@ucb.br}.}
}

\maketitle

\maketitle
\begin{abstract}
This paper is concerned with the Frank--Wolfe algorithm for a special class of {\it non-compact} constrained optimization problems. The notion of asymptotic cone is used to introduce this class of problems as well as to establish that the algorithm is well defined. These problems, with closed and convex constraint set, are characterized by two conditions on the gradient of the objective function. The first establishes that the gradient of the objective function is Lipschitz continuous, which is quite usual in the analysis of this algorithm. The second, which is new in this subject, establishes that the gradient belongs to the interior of the dual asymptotic cone of the constraint set. Classical results on the asymptotic behavior and iteration-complexity bounds for the sequence generated by the Frank--Wolfe algorithm are extended to this new class of problems. Examples of problems with non-compact constraints and objective functions satisfying the aforementioned conditions are also provided.
\end{abstract}



\section{Introduction}

The {\it Frank--Wolfe algorithm or conditional gradient method} is one of the oldest methods for finding minimizers of differentiable functions on compact convex sets. It was initially proposed in 1956 \cite{FrankWolfe1956} for solving quadratic programming problems with linear constraints (see also \cite{DemyanovRubinov1970,LevitinPolyak1966}), and it has since attracted considerable attention owing to its simplicity and ease of implementation, as it only requires access to a linear minimization oracle over the constraint set. In particular, it allows low storage cost and readily exploits separability and sparsity; therefore, it can be effectively applied to large-scale problems. With the emergence of machine-learning applications, this method has recently gained increasing popularity \cite{Jaggi2013, LacosteJaggi2015, Lan2013}. Accordingly, it has been extensively studied, and several variants thereof have been developed \cite{BeckTeboulle2004,HarchaouiNemirovski2015,BoydRecht2017,LussTeboulle2013,FreundMazumder2017,Konnov2018,Ghadimi2019, LanZhou2016} and references therein.

The aim of this study is to extend the Frank--Wolfe algorithm to a special class of constrained optimization problems $ {\Minimize}_{x\in {C}} f(x)$, where  $f: {\mathbb R}^n \to {\mathbb R}$ is a continuously differentiable function, and $C\subset {\mathbb R}^n$ is a closed and convex but not necessarily compact set. In addition to the classical assumptions (i.e., the gradient of $f$ is Lipschitz continuous), we assume that  $\nabla f(x)\in \mbox{int}(C_{\infty})^*$ for all $x\in C$, where $\nabla f$ and $\mbox{int}(C_{\infty})^*$ denote the gradient of $f$ and the interior of the positive dual asymptotic cone of $C$, respectively. For this class of functions, with classical assumptions, we also extend previous results on the asymptotic behavior and iteration-complexity bounds for the sequence generated by the Frank--Wolfe algorithm.

This paper is organized as follows. In Section~\ref{sec:Preliminares}, notations and auxiliary results are presented. In Section~\ref{secc3}, we formulate the Frank--Wolfe algorithm. In Section~\ref{sec:wd}, we establish that the sequence generated by this algorithm is well defined. Section~\ref{sec:aa} is devoted to the study of the asymptotic-convergence properties of Algorithm~\ref{Alg:CondG}, and Section~\ref{sec4} to the study of the iteration-complexity bounds. In Section~\ref{sec-Ap}, we present some examples. We conclude the paper in Section~\ref{Sec:Conclusions}.
\section{Preliminaries} \label{sec:Preliminares}
Herein, we present notations, definitions, and auxiliary results. Let ${\mathbb R}^n$ be the $n$-dimensional Euclidean space with the usual inner product $\langle\cdot,\cdot\rangle$ and norm $\|\cdot\|$. We denote by ${\mathbb R}^{m \times n}$ the set of all $m\times n$
matrices with real entries (${\mathbb R}^n\equiv {\mathbb R}^{n \times 1}$), by $e^i$ the $i$-th canonical unit vector in ${\mathbb R}^n$, and by ${\rm I_n}$ the $n\times
n$ identity matrix. A set ${\cal K}\subseteq{\mathbb R}^{n}$ is called a {\it cone} if for any $\alpha> 0$ and $x\in \cal{K}$, we have $\alpha x\in \cal{K}$. A cone
${\cal K}\subseteq{\mathbb R}^n$ is called {\it convex} if for any $x,y \in\cal{K}$, we have $x + y \in \cal{K}$. The {\it positive dual cone} of a cone ${\cal{K}} \subseteq
{\mathbb R}^n$ is the cone ${\cal{K}}^*\!\!:=\!\{ x\in {\mathbb R}^n :~ x^{\T}y \geq\! 0, ~ \forall \, y\!\in\! {\cal{K}}\}$, and its interior is denoted by $\mbox{int} {\cal{K}}^*\!\!:=\!\{ x\in {\mathbb R}^n :~ x^{\T}y > 0, ~ \forall \, y\!\in\! {\cal{K}} {\setminus\{0\} } \}$. Let $C\subset {\mathbb R}^n$ be a closed convex set; then, we define the {\it asymptotic cone} of $C$ by
$$
C_{\infty}:=\Big{\{}d\in {\mathbb R}^n: ~ \exists ~ (t_k)_{k\in\mathbb{N}}\subset (0, \infty), ~ \exists ~(x^k)_{k\in\mathbb{N}}\subset C; \lim_{k\to \infty}t_k=0, \lim_{k\to \infty}t_kx_k=d \Big{\}},
$$
or equivalently, $C_{\infty}:=\{ d\in {\mathbb R}^n: ~ x+td\in C, \forall ~x\in C,~ \forall ~ t\geq 0\} $ \cite[pp. 39]{UrrutyLemarechal2001}. Let $f:{\mathbb R}^n\to {\mathbb R}$ be a continuously differentiable function. We consider the problem of determining an {\it optimum point} of $f$ in a closed convex set ${C} \subset {\mathbb R}^n$, that is, a point $x^*\in {C}$ such that $f(x^*) \leq f(x)$ for all $x \in {C}$. We denote this constrained problem as
\begin{equation} \label{eq:mp}
\displaystyle {\Minimize}_{x\in {C}} f(x).
\end{equation}
 The optimal value of $f$ on ${C}$ is denoted by $f^*$, that is, $f^*:=\inf_{x\in {C}} f(x)$. The {\it first-order optimality condition } for problem~\eqref{eq:mp} is stated as
\begin{equation} \label{eq:oc}
 \nabla f(x^*)^{T}(x-x^*) \geq 0 , \qquad \forall~x\in {C}.
\end{equation}
In general, the condition \eqref{eq:oc} is necessary but not sufficient for optimality. Thus, a point $x^*\in C$ satisfying condition \eqref{eq:oc} is called a {\it stationary point} to problem~\eqref{eq:mp}.
\begin{definition} \label{def:sccf}
Let $C\subset {\mathbb R}^n$ be a convex set. A function $f:{\mathbb R}^n\to {\mathbb R}$ is called $M$-strongly convex with parameter $M\geq 0$ on $C$ if the inequality $f(tx+(1-t)y) \leq tf(x) + (1-t)f(y)- \frac{1}{2}Mt(1-t)\|y-x\|^2$ holds for all $t\in[0, 1]$ and $x,y \in {C}$. In particular, for $M=0$, $ f $ is called convex rather than $0$-strongly convex.
\end{definition}
The following results provide a useful characterization of convex/strongly convex differentiable functions, see the proof in \cite[Theorem 4.1.1, pp. 183]{UrrutyLemarechal1993_305}.
\begin{proposition} \label{pr:sccf}
 Let $C\subset {\mathbb R}^n$ be a convex set, $f:{\mathbb R}^n\to {\mathbb R}$ be a continuously differentiable function, and $M\geq 0$. Then, $f$ is $M$-strongly convex in $C$ if and only if $f(y)\geq f(x)+\nabla f(x)^T (y-x)+M\|x-y\|^2/2$ for all $x, y\in C$.
\end{proposition}
\begin{remark} It is well known that if $f$ is $M$-strong convex, then \eqref{eq:oc} is sufficient for optimality, that is, any point $x^*\in C$ satisfying \eqref{eq:oc} is a minimizer to problem~\eqref{eq:mp}.
\end{remark}
The proof of the next lemma can be found in \cite[Lemma~6]{polyak1987}.
\begin{lemma}\label{le:rate}
Let $\lbrace a_{k} \rbrace $	be a nonnegative sequence of real numbers. If $\Gamma a_{k}^2 \leq a_{k} - a_{k+1}$ for some $\Gamma >0$ and for any $k=1,...,\ell$, then $a_\ell \leq a_0/(1+\Gamma a_0\ell)< 1/(\Gamma \ell)$.
\end{lemma}
\section{Frank--Wolfe algorithm} \label{secc3}
Herein, we formulate the {\it Frank--Wolfe algorithm to solve problem~\eqref{eq:mp}}. To this end, {\it we henceforth assume that the constraint set $C\subset {\mathbb R}^n$ is closed and convex (not necessarily compact), the objective function $f:{\mathbb R}^n\to {\mathbb R}$ of problem~\eqref{eq:mp} is continuously differentiable, and its gradient satisfies the following condition:}
\begin{itemize}
\item[{\bf (A)}] $\| \nabla f(x)- \nabla f(y) \| \leq L\|x-y \|$ for all ~ $x,y \in {C}$ and $L>0$.
\end{itemize}
\begin{remark}
In Section~\ref{sec-Ap}, we present examples of problem~\eqref{eq:mp} with objective function satisfying {\bf (A)}.
\end{remark}
\noindent
{\it To formulate the Frank--Wolfe algorithm, we should assume that there exists a linear-optimization oracle (LO oracle) capable of minimizing linear functions over $C$}.
\bigskip
\hrule
\hrule
\begin{algorithm} \label{Alg:CondG}
 {\bf CondG$_{{C},f}$ method} \\
\hrule
\hrule
\begin{description}
\item[0.] Select $ x^0\in {C}$. Set $k=0$.
\item [1.] Use an ``LO oracle" to compute an optimal solution $p^k$ and the optimal value $v_{k}^*$ as
\begin{equation} \label{eq:CondG}
p^k \in {\argmin}_{p\in {C}} \nabla f(x^k)^T (p-x^k), \qquad v_{k}^*:= \nabla f(x^k)^T(p^k-x^k).
\end{equation}
\item[ 2.] If $v^*_{k}=0$, then {\bf stop}; otherwise, compute the step size $\lambda_k \in (0, 1]$ as
 \begin{equation}\label{eq:fixed.step}
\lambda_k:=\mbox{min}\left\{1, \frac{|v_{k}^*|}{L \|p^k-x^k\|^2}\right\}={\argmin}_{\lambda \in (0,1]}\left \{v_{k}^* \lambda+\frac{L }{2} \|p(x^k) -x^k\|^2 \lambda^2 \right \},
\end{equation}
 and set the next iterate $x^{k+1}$ as
\begin{equation}\label{eq:iteration}
x^{k+1}:=x^k+ \lambda_k(p^k-x^k).
\end{equation}
\item[ 3.] Set $k\gets k+1$, and go to step {\bf 1}.
\end{description}
\hrule
\hrule
\end{algorithm}

We conclude this section by stating a basic inequality for functions satisfying assumption~{\bf (A)} \cite[Lemma~2.4.2]{DennisSchnabel1996}.
\begin{lemma} \label{t:descent lemma}
	Let $f: \mathbb{R}^n \to \mathbb{R}$ be a continuously differentiable function satisfying condition~{\bf (A)}, $x \in {C}$, and $\lambda \in [0,1]$. Then,
	\begin{equation} \label{e:descent lemma}
f(x+\lambda(p-x)) \leq f(x) + \nabla f(x)^{\T}(p-x) \lambda + \frac{L}{2}\|p-x\|^2\lambda^{2}, \qquad \forall~p \in {C}.
	\end{equation}
\end{lemma}
\subsection{Well-definedness} \label{sec:wd}
Herein, we establish that the sequence $(x^k)_{k\in\mathbb{N}}$ generated by Algorithm~\ref{Alg:CondG} is well defined. To this end, {\it we assume that the gradient of the objective function $f:{\mathbb R}^n\to {\mathbb R}$ of problem~\eqref{eq:mp} satisfies the following condition:}
\begin{itemize}
\item[{\bf (B)}] $\nabla f(x)\in \mbox{int}(C_{\infty})^*$, for all $x\in C$.
\end{itemize}
\begin{remark}
It follows from \cite[Proposition 2.2.3]{UrrutyLemarechal2001} that a closed convex set $C\subset {\mathbb R}^n$ is compact if and only if $C_{\infty}=\{0\}$. Then, $\mbox{int}(C_{\infty})^*={\mathbb R}^n$, and therefore {\bf (B)} holds trivially if $C$ is compact. In Section~\ref{sec-Ap}, we present examples of problem~\eqref{eq:mp} with $\nabla f$ satisfying {\bf (B)} with unbounded constraint set $C$.
\end{remark}
We now use ${\bf (B})$ to prove a general result that implies the existence of a solution to the problem \eqref{eq:CondG}. It is worth mentioning that ${\bf (A})$ is not required in the proof.
\begin{proposition} \label{pr:wdc}
The following three assertions hold:
\begin{description}
\item[(i)] For each $x\in C$, the set $D_x:=\left\{p\in C:~ \nabla f(x)^T (p-x)\leq 0\right\}$ is compact.
\item[(ii)] For each $x\in C$, the linear problem
\begin{equation} \label{eq:lp}
\displaystyle {\Minimize}_{p\in {C}} \nabla f(x)^T (p-x)
\end{equation}
 has a solution.
\item[(iii)] If $D\subset C$ is a bounded set, then the set
\begin{equation} \label{eq:minbs}
\bigcup_ {x\in D}\Big\{q_x\in C:~ q_x\in {\argmin}_{p\in {C}} \nabla f(x)^T (p-x)\Big\},
\end{equation}
 is also bounded.
\end{description}
\end{proposition}
\begin{proof}
To prove (i), let $x\in C$. We assume toward a contradiction that $D_x$ is unbounded. Thus, let $(q^k)_{k\in\mathbb{N}}\subset D_x $ such that $\lim_{k\to \infty} \|q^k\|=\infty$, and let $(t_k)_{k\in\mathbb{N}}\subset (0, \infty) $ be the sequence defined by $t_k:=1/\|q^k\|$ for all $k=0,1\ldots $. Then, $ \lim_{k\to \infty}t_k=0$. As $t_k q^k=q^k/\|q^k\|$, we conclude that $\|t_kp_k\|=1$ for all $k=0, 1, \ldots$. Hence, there exist subsequences $(q^{k_j})_{j\in\mathbb{N}}\subset D_x $ and $(t_{k_j})_{j\in\mathbb{N}}\subset (0, \infty) $ such that $\lim_{k_j\to \infty}t_{k_j}q^{k_j}=d \in C_{\infty}$. Thus, the definition of $D_x$ implies
$$
 \nabla f(x)^T \left(t_{k_j}q^{k_j}- t_{k_j} x\right)=\nabla f(x)^T \left(\frac{q^{k_j}}{\|q^{k_j}\|}-\frac{x}{\|q^{k_j}\|}\right)\leq 0.
$$
Taking the limit in the last inequality as $j$ tends to $\infty$, we conclude that $ \nabla f(x)^Td\leq 0$, which is absurd. Indeed, assumption {\bf (B)} implies that  $\nabla f(x)\in \mbox{int}(C_{\infty})^*$, and as $ d \in C_{\infty}$, we have $ \nabla f(x)^Td> 0$. Therefore, (i) is proved. To prove (ii), it is sufficient to note that the problem \eqref{eq:lp} has $D_x$ as a sublevel set, which by (i) is compact. We now prove (iii). We assume toward a contradiction that the set in \eqref{eq:minbs} is unbounded. Then, there exist sequences $(x^{k})_{k\in\mathbb{N}}\subset D$ and $(q_{x^k})_{k\in\mathbb{N}}\subset C$ such that $\lim_{k\to \infty}\|q_{x^k}\|=\infty$. Thus, as $D$ is bounded, we have
\begin{equation} \label{eq:tklp}
\lim_{k\to \infty}\tau_k=0, \qquad \mbox{where} \quad \tau_k:= \frac{1}{\|q_{x^k}-x^k\|}, \qquad k=0,1, \ldots.
\end{equation}
However, as $C$ is convex, and $(x^k)_{k\in\mathbb{N}}$ and $\{q_{x^k}\}$ belong to $C$, we have
\begin{equation} \label{eq:cclp}
 x^k+ t\left(q_{x^k}-x^k\right) \in C, \qquad k=0,1, \ldots
\end{equation}
for any $t\in (0, 1)$. As $\tau_k (q_{x^k}-x^k)=(q_{x^k}-x^k)/\|q_{x^k}-x^k\|$, we have $\|\tau_k (q_{x^k}-x^k)\|=1$ for all $k=0, 1, \ldots$. Thus, there exist subsequences $(x^{k_j})_{j\in\mathbb{N}}\subset D$, $(q_{x^{k_j}})_{j\in\mathbb{N}}\subset C $, and $(t_{k_j})_{j\in\mathbb{N}}\subset (0, \infty) $ such that
\begin{equation} \label{eq:cclpc}
 \lim_{k_j\to \infty}\tau_{k_j}(q_{x^{k_j}}-x^{k_j})=v.
\end{equation}
As $(x^{k_j})_{j\in\mathbb{N}}\subset D$ and $D$ is bounded, \eqref{eq:tklp}, \eqref{eq:cclp}, and \eqref{eq:cclpc} yield
\begin{equation} \label{eq:crc}
 \lim_{k_j\to \infty}\tau_{k_j}\left[ x^{k_j} + t \left(q_{x^{k_j}}-x^{k_j}\right)\right]= \lim_{k_j\to \infty}\left[ \tau_{k_j}x^{k_j} + t \tau_{k_j} \left(q_{x^{k_j}}-x^{k_j}\right)\right]=tv \in C_{\infty}.
\end{equation}
As $q_{x^{k_j}}\in {\argmin}_{p\in {C}} \nabla f(x^{k_j})^T (p-x^{k_j})$ and $x^{k_j}\in C$, we have $ \nabla f(x^{k_j})^T (q_{x^{k_j}}-x^{k_j})\leq 0$ for all $j=0,1, \ldots$. Then, as $(t_{k_j})_{j\in\mathbb{N}}\subset (0, \infty) $, we have
$$
 \nabla f(x^{k_j})^T \left( \tau_{k_j}(q_{x^{k_j}}-x^{k_j})\right)\leq 0, \qquad j=0,1, \ldots.
$$
Considering that $D$ is bounded and $(x^{k_j})_{j\in\mathbb{N}}\subset D$, we can assume without loss of generality that $ \lim_{k_j\to \infty} x^{k_j}={\bar x}.$ Thus, taking the limit in the last inequality as $j$ tends to $\infty$ and using \eqref{eq:cclpc}, we obtain $ \nabla f({\bar x})^Tv\leq 0$, which is absurd because, by \eqref{eq:crc}, we have $ v \in C_{\infty}$, and by assumption  {\bf (B)}, we have $\nabla f({\bar x})\in \mbox{int}(C_{\infty})^*$. Therefore, the proof of (iii) is complete.
\end{proof}
In the next lemma, we establish that the sequence $(x^k)_{k\in\mathbb{N}}$ generated by Algorithm~\ref{Alg:CondG} is well defined. We also obtain results related to the optimal value $v_{k}^*$ defined in \eqref{eq:CondG}.
\begin{lemma} \label{eq;wd}
The sequences $\{p^k\}_{k\in\mathbb{N}} \subset C$ and $\{x^k\}_{k\in\mathbb{N}} \subset C$ are well defined. Moreover, the following assertions hold:
\begin{description}
\item[{(i)}] $v_{k}^* \leq 0$ for all $k=0, 1, \ldots$;
\item[{(ii)}] $v_{k}^*=0$ if and only if $x^k$ is a stationary point of problem~\eqref{eq:mp};
\item[{(iii)}] $v_{k}^*< 0$ if and only if $\lambda_k>0$ and $p^{k}\neq x^k$.
\end{description}
\end{lemma}
\begin{proof}
For each $x^k\in C$, it follows from Proposition~\ref{pr:wdc} that $p^ k$ and $v_{k}^*$ in \eqref{eq:CondG} can be computed and $p^ k\in C$. As $x^0\in C$ and $0\leq \lambda\leq 1$, by using \eqref{eq:iteration} and an inductive argument, we conclude that  $(p^k)_{k\in\mathbb{N}}$, $(v_{k}^*)_{k\in\mathbb{N}}$, and $(x^k)_{k\in\mathbb{N}}$ are well defined, and that $(p^k)_{k\in\mathbb{N}}$ and $(x^k)_{k\in\mathbb{N}}$ belong to $C$. To prove (i), it suffices to note that the optimality in \eqref{eq:CondG} implies $v_{k}^*\leq \nabla f(x^k)^T (x^k-x^k)=0$ for all $k=0, 1, \ldots$. To prove(ii), we note that \eqref{eq:CondG} implies that $v_{k}^*\leq \nabla f(x^k)^T (p-x^k)$ for all $p\in C$. Thus, if $v_{k}^*=0$, we conclude that $0\leq \nabla f(x^k)^T (p-x^k)$ for all $p\in C$. Hence, $x^k$ satisfies \eqref{eq:oc}, that is, $x^k$ is a stationary point of problem~\eqref{eq:mp}. Conversely, if $x^k$ is a stationary point of problem~\eqref{eq:mp}, then  \eqref{eq:oc} implies $0\leq \nabla f(x^k)^T (p-x^k)$ for all $p\in C$. As $p^ k\in C$, we conclude that $0\leq \nabla f(x^k)^T (p^k-x^k)=v_{k}^*$. Therefore, by (i), we have $v_{k}^*=0$. We now prove (iii). It is immediate from (i) and \eqref {eq:fixed.step} that $v_{k}^*< 0$ if and only if $\lambda_k>0$ and $p^{k}\neq x^k$; this concludes the proof.
\end{proof}

It follows from Lemma~\ref{eq;wd} that Algorithm~\ref{Alg:CondG} generates either an infinite sequence or a finite sequence $\{x^k\}_{k\in\mathbb{N}} \subset C$, the last iterate of which is a stationary point of problem~\eqref{eq:mp}. {\it Henceforth, let $\{p^k\}_{k\in\mathbb{N}} \subset C$ and $\{x^k\}_{k\in\mathbb{N}} \subset C$ be sequences generated by Algorithm~\ref{Alg:CondG}; we assume that these sequences are infinite.}
\subsection{Asymptotic convergence analysis}\label{sec:aa}
Herein, we study the asymptotic convergence of Algorithm~\ref{Alg:CondG}. We first prove an important inequality.
\begin{lemma} \label{le:dss}
The following inequality holds: 	
\begin{equation} \label{e.sufficient decrease}
f(x^{k+1}) \leq f(x^{k}) - \frac{1}{2} |v_{k}^*| \lambda_k, \qquad k=0,1, \ldots.
\end{equation}
Consequently, $f(x^k) > f(x^{k+1})$ for all $k=0,1, \ldots$.
\end{lemma}
 \begin{proof}
Let $x^k\in {C}$ be defined as in Algorithm~\ref{Alg:CondG}, and $v_{k}^*$ as in \eqref{eq:CondG}. We first recall that we have assumed that $(x^k)_{k\in\mathbb{N}}$ is infinite. Thus, Lemma~\ref{eq;wd} implies that  $v_{k}^*< 0$ and $p^k\neq x^k$. Applying Lemma~\ref{t:descent lemma} with $x = x^k$, $p=p^k$, and $\lambda = \lambda_k$, we have
\begin{equation} \label{e.f.i}
f(x^{k+1}) \leq f(x^k) + v_{k}^* \lambda_k + \frac{L}{2}\|p^k-x^k\|^2\lambda_k^{2}.
\end{equation}
We separately consider two cases: $\lambda_k= |v_{k}^*|/{(L \|p^k-x^k\|^2)}$ and $\lambda_k = 1$. In the former, it follows from \eqref{e.f.i} that
\begin{equation} \label{eq:fifr}
f(x^{k+1}) \leq f(x^{k}) - \frac{1}{2} |v_{k}^*| \lambda_k.
\end{equation}
If now $\lambda_k = 1$, then \eqref{e.f.i} becomes $f(x^{k+1}) \leq f(x^k) - |v_{k}^*| + L\|p^k-x^k\|^2/2$, and \eqref{eq:fixed.step} yields $\lambda_k=1 \leq |v_{k}^*|/(L \|p^k-x^k\|^2)$. Thus, we obtain
$ f(x^{k+1}) \leq f(x^k) - (|v_{k}^*|/2)\lambda_k.$
Therefore, combining this inequality with \eqref{eq:fifr} yields \eqref{e.sufficient decrease}. As we assumed that $v_{k}^*< 0$ for all $k=0, 1, \ldots$, the second part follows, and the proof is complete.
 \end{proof}
The next result shows a partial asymptotic-convergence property of the Frank--Wolfe algorithm; it requires neither convexity nor strong convexity on $f$.

\begin{theorem} \label{th:gr}
Each limit point ${\bar x}\in C$ of $(x^k)_{k\in\mathbb{N}}$ is stationary for the problem~\eqref{eq:mp}.
\end{theorem}
\begin{proof}
Let ${\bar x}\in C$ be a limit point of the sequence $(x^k)_{k\in\mathbb{N}}$, and let $(x^{k_j})_{j\in\mathbb{N}}$ be a subsequence of $(x^k)_{k\in\mathbb{N}}$ such that $\lim_{j\to \infty}x^{k_j}={\bar x}$. Hence, $\lim_{j\to \infty}f(x^{k_j})=f({\bar x})$. As Lemma~\ref{le:dss} implies that $(f(x^k))_{k\in\mathbb{N}}$ is a decreasing sequence, we conclude that $(f(x^k))_{k\in\mathbb{N}}$ converges to $f({\bar x})$. Thus, in particular,
\begin{equation} \label{e.limite}
\lim_{ k \to \infty} [f(x^k)-f(x^{k+1})] = 0.
\end{equation}
Moreover, as $(x^{k_j})_{j\in\mathbb{N}}$ is bounded, by combining the inclusion in \eqref{eq:CondG} with (iii) of Proposition~\ref{pr:wdc}, it follows that $(p^{k_j})_{j\in\mathbb{N}}$ is also bounded. Let $(p^{k_\ell})_{\ell \in\mathbb{N}}$ be a subsequence of $(p^{k_j})_{j\in\mathbb{N}}$ such that $\lim_{\ell\to \infty}p^{k_\ell}={\bar p}$. If ${\bar x}={\bar p}$, then by \eqref{eq:CondG} and the continuity of $\nabla f$, we have $\lim_ {\ell \to \infty }v_{k_\ell}^*= \nabla f(x^{k_\ell})^T(p^{k_\ell}-x^{k_\ell})=0$. We now assume that ${\bar x}\neq {\bar p}$. Hence, combining Lemma~\ref{le:dss} with \eqref{e.limite}, we conclude that $\lim_{ \ell \to \infty} |v_{k_\ell}^*| \lambda_{k_\ell}=0$, whereas \eqref{eq:fixed.step} yields
\begin{equation} \label{eq:sew}
 |v_{k_\ell}^*| \lambda_{k_\ell}= \mbox{min}\left\{ |v_{k_\ell}^*|, \frac{|v_{k_\ell}^*|^2}{L \|p^{k_\ell}-x^{k_\ell}\|^2}\right\}.
\end{equation}
As $\lim_{\ell\to \infty}x^{k_\ell}={\bar x}$, $\lim_{\ell\to \infty}p^{k_\ell}={\bar p}$, and ${\bar x}\neq {\bar p}$, we obtain that $ \lim_{ \ell \to \infty} \|p^{k_\ell}-x^{k_\ell}\|=\|{\bar x}- {\bar p}\|\neq 0$. Thus, as $\lim_{ \ell \to \infty} |v_{k_\ell}^*| \lambda_{k_\ell}=0$, it follows from \eqref{eq:sew} that $\lim_{ \ell \to \infty} |v_{k_\ell}^*| =0$. Moreover, the optimality of $v_{k_\ell}^*$ in \eqref{eq:CondG} implies
\begin{equation} \label{e.subseq}
v_{k_\ell}^* \leq \nabla f(x^{k_\ell})^{\T}(p-x^{k_\ell}), \qquad \forall~ p\in C.
\end{equation}
As $\lim_{ \ell \to \infty} v_{k_\ell}^* = 0$, taking the limit in \eqref{e.subseq} and using the continuity of $\nabla f$, we have $\nabla f(\bar{x})^{\T}(p - \bar{x}) \geq 0$ for all $p\in C$. Therefore, $\bar{x}$ is stationary for the problem~\eqref{eq:mp}.
\end{proof}
We now show that if $f$ is assumed convex, we can improve the previous result.
\begin{theorem} \label{th:fconvr}
The following assertions hold:
\begin{description}
	\item[{(i)}] If $f$ is a convex function and $x^*$ is a cluster point of the sequence $(x^k)_{k\in\mathbb{N}}$, then $x^*$ is a solution for the problem~\eqref{eq:mp}.
	\item[{(ii)}] If $f$ is an $M$-strongly convex function, then $(x^k)_{k\in\mathbb{N}}$ converges to a point $x^*\in C$ that is a solution of problem~\eqref{eq:mp}. Moreover, $\|x^k-x^*\|\leq \sqrt{2(f(x^k)-f(x^*))/M}$ for all $k=0, 1,\ldots$.
\end{description}	
\end{theorem}
\begin{proof}
To prove(i), we assume that $f$ is convex and $x^*$ is a cluster point of $(x^k)_{k\in\mathbb{N}}$. As $f$ is convex, applying Proposition~\ref{pr:sccf} with $M=0$, we obtain that $f(p)\geq f(x^*)+\nabla f(x^*)^T (p-x^*)$ for all $p\in C$. Therefore, considering that Theorem~\ref{th:gr} implies that $\nabla f(x^*)^{\T}(p - x^*) \geq 0$ for all $p\in C$, we conclude that $f(p) \geq f(x^*)$ for all $p\in C$. Then, $x^*$ is a solution for the problem~\eqref{eq:mp}. To prove (ii), we first note that as $f$ is $M$-strongly convex, the level set ${\cal L}_{f(x^0)}:=\{x\in C: ~f(x)\leq f(x^0)\}$ is bounded. Lemma~\ref{le:dss} now implies that $(x^k)_{k\in\mathbb{N}}\subset {\cal L}_{f(x^0)}$. Hence, $(x^k)_{k\in\mathbb{N}}$ is also bounded. Let $x^*\in C$ be a cluster point of $(x^k)_{k\in\mathbb{N}}$. It follows from (i) that $x^*$ is a solution of problem~\eqref{eq:mp}. Furthermore, combining \eqref{eq:oc} with Proposition~\ref{pr:sccf}, we obtain
 \begin{equation} \label{eq:isc}
 f(x^k)-f(x^*)\geq \frac{M}{2}\|x^*-x^k\|^2, \qquad k=0, 1,\ldots.
\end{equation}
As Lemma~\ref{le:dss} implies that $(f(x^k))_{k\in\mathbb{N}}$ is a decreasing sequence, we have  $\lim_{k \to +\infty}f(x^k)=f(x^*)$. Therefore, \eqref{eq:isc} implies that $(x^k)_{k\in\mathbb{N}}$ converges to $x^* $. Finally, we note that \eqref{eq:isc} is equivalent to the inequality in (ii), and the proof is complete.
\end{proof}
\subsection{Iteration-complexity analysis}\label{sec4}
Herein, we derive two iteration-complexity bounds for the sequence $(x^k)_{k\in\mathbb{N}}$ generated by Algorithm~\ref{Alg:CondG}. To this end, we assume that $\lim_{ k \to \infty} x_{k} = x^*$. Then, (iii) of Proposition~\ref{pr:wdc} implies that $(p^k)_{k\in\mathbb{N}}$ is  bounded. Therefore, we define
\begin{equation} \label{eq;bnn}
0< \sigma:=\sup_{k}\{ \|p^k-x^k\|: ~k=0, 1, \ldots\}<\infty.
\end{equation}
We also define the following constants:
\begin{equation} \label{eq:Const}
\Gamma:=\min \left\{\frac{1}{2\gamma \sigma},\frac{1}{2L\sigma^2}\right\} > 0, \qquad \gamma:=\max \left\{\|\nabla f(x^k)\|: ~k=0, 1, \ldots \right\}>0.
\end{equation}
\begin{theorem} \label{t:convergence CGM}
The following assertions hold:
\begin{description}
	\item[{\bf (i)}] If $f$ is a convex function, then $f(x^k) - f^{*} \leq \Gamma^{-1}/k$ for all $k=1, 2,\ldots$.
	\item[{\bf (ii)}] If $f$ is $M$-strongly convex, then $\|x^k-x^*\|\leq \sqrt{2/(\Gamma M)}/\sqrt{k}$ for all $k=1, 2,\ldots$.
\end{description}
\end{theorem}
\begin{proof}
To prove (i), we first prove the following inequality:
\begin{equation} \label{eq:lbci}
	 \Gamma {v_{k}^*}^2\leq f(x^k)-f(x^{k+1}), \qquad \forall~ k=0, 1,\dots,
	\end{equation}	
where $\Gamma$ is defined in \eqref{eq:Const}. By using Lemma~\ref{le:dss}, with $\sigma$ defined in\eqref{eq;bnn}, and considering that ${v_{k}^*} < 0$,  we conclude after some algebraic manipulations that
\begin{equation} \label{e.sufdec}
  \min\left\{ \frac{1}{2|v_{k}^*|} , \frac{1}{2 L \sigma^2}\right\} {v_{k}^*}^2 \leq f(x^{k}) -f(x^{k+1}), \qquad \forall~ k=0, 1,...
\end{equation}
Further, by combining \eqref{eq:CondG} with  \eqref{eq;bnn} and the second equality in \eqref{eq:Const}, we obtain $0< |v_{k}^*|\leq \| \nabla f(x^k)\| \|x^k-p^k\|\leq \gamma \sigma$ for all $k=0, 1, \ldots$, which implies
$$
\frac{1}{ \gamma \sigma}\leq \frac{1}{ |v_{k}^*|}, \qquad k=0, 1, \ldots.
$$
Thus, \eqref{eq:lbci} follows from \eqref{e.sufdec}, the previous inequality, and \eqref{eq:Const}. It now follows from (ii) of Theorem~\ref{th:fconvr} that $x^*$ is a solution of problem~\eqref{eq:mp}. As $f$ is convex, we have $f^*=f(x^*) \geq f(x^k)+ \nabla f(x^k)^{T}(x^{*} - x^k)$ for all $k$. Thus, \eqref{eq:CondG} implies that $f^* - f(x^k) \geq \nabla f(x^k)^{T}(x^{*} - x^k) \geq f(x^k)^{T}(p^{k} - x^k) = v_{k}^*$ for all $k$. As $ f^* \leq f(x^k)$ for all $k$, we conclude that $ v_{k}^* \leq f^* - f(x^k) \leq 0$ for all $k$. Therefore, we obtain
$ (f(x^k)-f^* )^{2} \leq v_{k}^{*2} $ for all $k=0, 1, \ldots$; this, combined with \eqref{eq:lbci}, yields
 \begin{equation} \label{e.72}
 (f(x^k)-f^* )-(f(x^{k+1})-f^*) \geq \Gamma (f(x^k)-f^* )^{2}, \qquad k=0, 1, \ldots.
 \end{equation}
As $\Gamma > 0$, if we define $a_k :=f(x^k)-f^*$, we conclude from \eqref{e.72} that $a_k-a_{k+1} \geq \Gamma a_k^{2}$. Thus, (i) follows by applying Lemma~\ref{le:rate}. To prove (ii), we first note that the $M$-strong convexity of $f$ and \eqref{eq:oc} imply that $\|x^k-x^*\|\leq \sqrt{2(f(x^k)-f(x^*))/M}$ for all $k=0, 1,\ldots$. Therefore, (ii) follows by using the inequality in (i), and the proof is complete.
 \end{proof}

\section{Examples}\label{sec-Ap}
Herein, we present examples of problem~\eqref{eq:mp}, with $f$ satisfying {\bf (A)} and {\bf (B)}, and $C$ unbounded. We first present a general class of functions $f$ satisfying {\bf (A)} and {\bf (B)}. To this end, let ${\cal{K}}\subset \R^n$ be a closed convex cone such that $\mbox{int} {\cal{K}} \cap \mbox{int} {\cal{K}}^* \neq \varnothing$,
$G : \R^n \to \R^n$ be a differentiable function, and $G'$ be its Jacobian. We assume that for constants $L_1>0$ and $L_2>0$, the function $G$ satisfies the following conditions:
\begin{description}
\item[{\bf (C1)}]$\| G(x)- G(y) \| \leq L_1\|x-y \|$ for all $x,y \in {\cal{K}}$;
\item[{\bf (C2)}]$\| G'(x)x- G'(y)y \| \leq L_2\|x-y \|$ for all $x,y \in {\cal{K}}$;
\item[{\bf (C3)}]$G(x)+G'(x)x\in {\cal{K}}^*$ for all $x \in {\cal{K}}$.
\end{description}
Let $a\in \mbox{int} {\cal{K}}^*$, and a function $f : \R^n \to \R$ be defined by
\begin{equation} \label{eq:hth}
f(x) :=a^Tx + G(x)^Tx.
\end{equation}
\begin{lemma} \label{le:ge}
Let $C\subset {\cal{K}}$ be closed and convex. Then, the gradient $\nabla f$ of $f$ defined in \eqref{eq:hth} satisfies conditions {\bf (A)} and {\bf (B)} on $C$.
\end{lemma}
\begin{proof}
We first note that the gradient $\nabla f$ of the function $f$ defined in \eqref{eq:hth} is given by
\begin{equation} \label{eq:dhth}
\nabla f(x)=a+ G(x)+G'(x)x.
\end{equation}
Thus, using {\bf (C1)} and {\bf (C2)}, we conclude that $\| \nabla f(x)- \nabla f(y) \| \leq (L_1+L_2)\|x-y \|$ for all $x,y \in C$. Hence, $\nabla f$ satisfies {\bf (A)}. Finally, as $a\in \mbox{int} {\cal{K}}^*$, {\bf (C3)} implies that $\nabla f(x)\in \mbox{int} {\cal{K}}^*$ for all $x\in C$. Furthermore, as $C\subset {\cal{K}}$, we obtain that $C_\infty \subset {\cal{K}}$. Hence, $ {\cal{K}}^*\subset C_\infty^*$, which implies that  $ \mbox{int} {\cal{K}}^* \subset \mbox{int}(C_\infty^*)$. Thus, we conclude that $\nabla f(x)\in \mbox{int}(C_{\infty})^*$ for all $x\in C$. Therefore, $\nabla f$ satisfies {\bf (B)} on $C$.
\end{proof}
We now recall a well-known result about Lipschitz functions, which is an immediate consequence of the mean-value inequality.
\begin{theorem} \label{th:csl}
Let $C\subset {\mathbb R}^n$ be a convex set, $F:{\mathbb R}^n\to {\mathbb R}^m$ be a continuously differentiable function, and $F'$ be its Jacobian. We assume that there exists a constant $L\geq 0$ such that $\|F'(x)\|\leq L$ for all $x\in C$. Then, $F$ is Lipschitz continuous with constant $L$ on $C$, that is, $\| F(x)- F(y) \| \leq L\|x-y \|$ for all $x,y \in C$.
\end{theorem}
Moreover, the following characterization of convex functions is required; its proof can be found in \cite[Theorem 4.1.1, pp. 190]{UrrutyLemarechal1993_305}.
 \begin{theorem}\label{th:csc}
Let $C\subset {\mathbb R}^n$ be a convex set, $f:{\mathbb R}^n\to {\mathbb R}$ be a twice continuously differentiable function, $\nabla^2 f$ be its Hessian, and $M\geq 0$ be a constant. Then, $f$ is $M$-strongly convex on $C$ if and only if $v^T\nabla^2 f(x)v\geq M\|v\|^2$ for all $x\in C$ and $v\in {\mathbb R}^n$.
\end{theorem}
In the following, we present specific examples of functions $G$ satisfying {\bf (C1)}--{\bf (C3)} on the cone ${\cal{K}}= \R^n_{+}$ such $f$ in \eqref{eq:hth} is convex and $\nabla f$ satisfies {\bf (A)} and {\bf (B)} for any closed convex set $C\subset \R^n_{+}$.
 \begin{example}
 Let $Q=(q_{i j} )\in {\mathbb R}^{n \times n}$ with entries $q_{i j} \geq 0$ for all $i, j$ and $a\in \R^n_{++}$. Let $G: \R^n \to \R^n$ be a linear function defined by $G(x)= Qx$. Direct calculations show that $\| G(x)- G(y) \|=\| G'(x)x- G'(y)y \| \leq \|Q\|\|x-y \|$ for all $x,y \in\R^n$. Thus, $G$ satisfies {\bf (C1)} and {\bf (C2)}, with $L_1=L_2=\|Q\|$, for any cone ${\cal{K}}$. As $q_{i j} \geq 0$ and $a\in \R^n_{++}$, $G$ also satisfies {\bf (C3)} in ${\cal{K}}= \R^n_{+}$. Moreover, if $Q$ satisfies the condition $v^TQv\geq M\|v\|^2$ for all $v\in {\mathbb R}^n$ and some $M\geq 0$, then Theorem~\ref{th:csc} implies that the problem~\eqref{eq:mp} with the associated quadratic function $f(x):=a^Tx + x^TQx$ is $M$-strong convex. Moreover, Lemma~\ref{le:ge} implies that $\nabla f(x)=a+Qx$ satisfies {\bf (A)} and {\bf (B)} for any closed and convex set $C\subset \R^n_{+}$. For instance, $C:=\left\{x\in\mathbb{R}^p:~x_1\geq x_2\geq\cdots\geq x_p\geq 0\right\}$, the monotone nonnegative cone.
\end{example}
 \begin{example}
Let $e:=(1, \ldots, n)\in \R^n$ be a vector, and $\alpha >0$ and $\beta >0$ be constants satisfying $2\alpha > 3\beta^{3/2}\sqrt{n}$. Then, the function $G: \R^n \to \R^n$ defined by
\begin{equation} \label{eq:fga}
G(x):= \alpha x + \frac{\beta}{ \sqrt{1+\beta x^T x}}e,
\end{equation}
satisfies {\bf (C1)}--{\bf (C3)} in ${\cal{K}}= \R^n_{+}$. Moreover, the problem~\eqref{eq:mp} with the associated function
\begin{equation} \label{eq:af}
f(x):=a^Tx + \alpha x^Tx + \frac{\beta}{ \sqrt{1+\beta x^T x}}e^Tx ,
\end{equation}
with $a\in \R^n_{++}$ is $M$-strong convex. Moreover, $\nabla f$ satisfies {\bf (A)} and {\bf (B)}. Indeed, we first note that $f(x) :=a^Tx + G(x)^Tx$. Some calculations show that
\begin{equation} \label{eq:dfda}
G'(x)= \alpha{\rm I_n}+ \frac{-\beta^2}{(1+\beta x^T x)^{3/2}}ex^T, \qquad G(x)+ G'(x)x= 2\alpha x+ \frac{\beta}{(1+\beta x^T x)^{3/2}}e,
\end{equation}
where ${\rm I_n}\in {\mathbb R}^{n \times n}$ is the identity matrix. Thus, $\nabla f (x)= a+ G(x)+ G'(x)x$. Hence, after some calculations, we have
\begin{equation} \label{eq:dhda}
\nabla^2 f(x)= 2\alpha{\rm I_n}+ \frac{-3\beta^2}{(1+\beta x^T x)^{5/2}}ex^T.
\end{equation}
The first equality in \eqref{eq:dfda} yields $\|G'(x)\|\leq \beta^{3/2} \sqrt{n}+ \alpha$ for all $x\in C$, and \eqref{eq:dhda} implies
\begin{equation} \label{eq:dhpd}
 0<\left(2\alpha -3\beta^{3/2}\sqrt{n}\right)|v\|^2\leq v^T \nabla^2 f(x)v\leq \left(2\alpha+ 3\beta^{3/2}\sqrt{n}\right)\|v\|^2,
\end{equation}
for all $x\in C$ and all $v\in {\mathbb R}^n$. As $\|G'(x)\|\leq \beta^{3/2}\sqrt{n}+ \alpha$, it follows from Theorem~\ref{th:csl} that $G$ also satisfies {\bf (C1)} with $L_1=\beta^{3/2}\sqrt{n}+ \alpha$. In particular, \eqref{eq:dhpd} implies that $\|\nabla^2f(x)\|\leq 2\alpha+ 3\beta^{3/2}\sqrt{n}$, for all $x\in C$. Hence, as $\|G'(x)\|\leq \beta^{3/2} \sqrt{n}+ \alpha$ for all $x\in C$, Theorem~\ref{th:csl} also implies that
$$
\| G'(x)x- G'(y)y \| \leq \|G(x)-G(y)\| + \|\nabla f(x)-\nabla f(y)\|\leq \left(4 \beta^{3/2}\sqrt{n}+ 3\alpha\right)\|x-y\|,
$$
for all $x, y\in C$. Thus, $G$ also satisfies {\bf (C2)} with $L_2=4 \beta^{3/2}\sqrt{n}+ 3\alpha$.
The second inequality in \eqref{eq:dfda} implies that $G(x)+G'(x)x\in {\R^n_{++}}$ for all $x \in \R^n_{+}$. As $(\R^n_+)^*=\R^n_+$, $G$ satisfies {\bf (C3)}. Therefore, Lemma~\ref{le:ge} implies that $\nabla f$ of $f$ in \eqref{eq:af} satisfies conditions {\bf (A)} and {\bf (B)} for any closed and convex set $C\subset \R^n_{+}$. For instance, $C:=\left\{x\in\mathbb{R}^p:~x_1\geq x_2\geq\cdots\geq x_p\geq 0\right\}$, the monotone nonnegative cone. Finally, using \eqref{eq:dhpd}, it follows from Theorem~\ref{th:csc} that $f$ in \eqref{eq:af} is $M$-strong convex with $M= 2\alpha -3\beta^{3/2}\sqrt{n}$.
\end{example}
In the next example, we present directly a convex function $f$ satisfying {\bf (A)} and {\bf (B)} in a closed convex set $C\subset \R^n_{+}$
\begin{example}
Let $\beta >0$, $a\in \R^n_{++}$, and $f: \R^n \to \R$ be defined by
\begin{equation} \label{eq:fsq}
f(x):=a^T x+ \sqrt{1+\beta x^T x}. 
\end{equation}
Let $C:=\left\{x\in\mathbb{R}^p:~x_1\geq x_2\geq\cdots\geq x_p\geq 0\right\}$ be the monotone nonnegative cone. 
 We note that, in this case, the gradient and the Hessian of $f$ are given by
$$
\nabla f(x)=a+\frac{\beta}{ \sqrt{1+\beta x^T x}}x, \qquad \nabla^2 f(x)= \frac{\beta}{ \sqrt{1+\beta x^T x}}{\rm I_n} - \frac{\beta^2}{(1+\beta x^T x)^{3/2}}xx^T,
$$
respectively, where ${\rm I_n}\in {\mathbb R}^{n \times n}$ is the identity matrix. Some calculations show that
$$
 \frac{\beta}{(1+\beta x^T x)^{3/2}}v^Tv \leq v^T\nabla^2 f(x)v\leq \beta v^Tv, \qquad \forall~v\in \R^n,
$$
which implies that $\nabla^2 f(x)$ is positive definite and $\|\nabla^2 f(x)\|\leq \beta$. Thus, using Theorems~\ref{th:csc} and \ref{th:csl}, we conclude that $f$ is convex, and $\nabla f$ is Lipschitz continuous with constant $\beta$. Moreover, for any closed and convex set $C \subset \R^n_{+}$, we have $C_\infty \subset \R^n_+$. Hence, $\R^n_+ =(\R^n_+)^*\subset C_\infty^*$, which implies that $\R^n_{++} \subset int(C_\infty^*)$. As $\nabla f(x)\in \R^n_{+ +}$ for all $x\in C$, we conclude that $\nabla f(x) \in int(C_\infty^*)$ for all $x\in C$. Finally, the convexity of $g$ implies that $C$ is convex. Therefore, the problem~\eqref{eq:mp} with the objective function \eqref{eq:fsq} is convex, and $f$ satisfies conditions {\bf (B)} and {\bf (A)}.
\end{example}
Let us present two more examples of convex functions and the respective convex sets satisfying conditions {\bf (B)} and {\bf (A)}.
\begin{description}
\item[(i)] Let $f: \R^n \to \R$ be defined by $f(x)=\ln(e^{x_1}+e^{x_2}+\ldots+ e^{x_n)}$. This function satisfies {\bf (A)} \cite[Example 5.15, pp. 115]{Beck2017}. Some calculations show that it also satisfies {\bf (B)} for any convex set $C\subset \R^n_{+}$.
\item[(ii)] Let $C\subset \R^n_{++}$ be closed and convex, and $d_{C}(x):=\min_{y\in C}\|x-y\|$ for $x\in \R^n$. We define the convex function $\psi_C: \R^n \to \R$ by $\psi_C(x):=\frac{1}{2}\|x\|^2-\frac{1}{2} d^2_{C}(x)$ \cite[Example 2.17.4, pp. 22]{Beck2017}. We can prove that $\nabla \psi_C(x)=P_C(x)$, where $P_C$ denotes the orthogonal projection onto $C$ \cite[Example 3.49, pp. 61]{Beck2017}. As $\R^n_{++} \subset int(C_\infty^*)$ and $C\subset \R^n_{++}$, we conclude that $\nabla \psi_C(x) \in C_\infty^*$, which implies that $\psi_C $ satisfies {\bf (B)}. Moreover, the nonexpansivity of the projection implies that $\psi_C $ also satisfies {\bf (A)}.
\end{description}
We conclude this section by presenting examples of unbounded convex sets that appear as constraints in optimization problems.
\begin{description}
\item [(a)] $C:=\{ x\in \R^n_{+}:~ 1\leq x_1\ldots x_n\}$;
\item [(b)] $C:=\{ x\in \R^n_{+}:~ 1\leq x_1+\ldots+ x_n\}$;
\item [(c)] $C:=\{ x\in \R^n_{+}:~ b\leq Ax\}$, where $A=(a_{i j} )\in {\mathbb R}^{n \times n}$ with $a_{i j} > 0$ and $b\in \R^n_{++}$.
\end{description}
Let $\Omega \subset \R^n$ be a closed, convex set, and $g: \Omega \to \R$ be a convex function. The epigraph of $g$ is defined by
$
\mbox{epi}(g):=\{ (x, t)\in \Omega \times \R:~g(x)\leq t\}.
$
The set $C=\mbox{epi}(g)\subset {\mathbb R}^{n}\times {\mathbb R}$ is convex and unbounded. We now provide specific examples.
\begin{description}
\item[(1)] (Lorentz cone) $\left\{ (x, t)\in \R^n \times \R:~\|x\|_2\leq t\right\}$, where $\|\cdot \|_2$ denotes the $2$-norm;
\item[(2)] $\left\{ (x, t)\in \R^n \times \R:~\|x\|_1\leq t\right\}$, where $\|\cdot \|_1$ denotes the $1$-norm;
\item[(3)] $\left\{ (x_1, \ldots, x_{n}, t) \in \ \R^{n}_{++} \times \R:~1/(x_1\ldots x_{n})\leq t\right\}$.
\end{description}
We point out that projecting on the sets $(1)$, $(2)$, and $(3)$ is not a particularly expensive task \cite[Chapter~6]{Beck2017}. Finally, we present an example on the cone of positive semidefinite matrices. 
\begin{example}
The cone of positive semidefinite (resp., definite) $n \times n$
symmetric matrices is denoted by ${\mathbb S}^{n}_{+}$ (resp., ${\mathbb S}^{n}_{++}$) and is selfdual, that is, ${{\mathbb S}^{n}_{+}}^*={\mathbb S}^{n}_{+}$. The trace of $X=(X_{ij}) \in {\mathbb S}^{n}$ is denoted by $tr X := \sum_{i=1}^p X_{ii}$. Given $X$ and $Y$ in ${\mathbb S}^{n}$, their inner product is defined as $\langle X, Y\rangle := tr X Y = \sum_{i=1,j=1}^{n,m} X_{ij}Y_{ij}$, whereas the norm of $X $ is defined by $\|X\| := \langle X, X\rangle^{1/2}$. Let $\beta >0$ and $\gamma>0$ be such that $2\beta>\gamma$, and let $f: {\mathbb S}^{n}_{+}\to \R$ be defined by $f(X):=\beta \|X||^2+\gamma\ln(1+tr X).$
We consider the set $C:=\{ X\in {\mathbb S}^{n}_{+}:~ g(X)\leq 0\} $, where $g: {\mathbb S}^{n}_{+}\to \R$ is a convex function. The gradient of $f$ is given by 
$$
\nabla f(X)=2\beta X + \frac{1}{1+tr X} {\rm I_n} \in {\mathbb S}^{n}_{++}, \qquad \forall ~X\in {\mathbb S}^{n}_{+}.
$$
Calculating the Hessian of $F$, we obtain
$$
\langle \nabla^2 f(X)V, V\rangle=2\beta \|V\|^2 - \frac{\gamma}{\left(1+tr X\right)^2}(tr V)^2, \quad \forall ~X\in {\mathbb S}^{n}_{+}, \quad \forall ~V\in {\mathbb S}^{n}.
$$
Moreover,  
$0<(2\beta -\gamma)\|V\| ^2\langle \nabla f(X)V, V\rangle\leq 2\beta \|V\| ^2$ for all $X\in {\mathbb S}^{n}_{+}$ and $0\neq V\in {\mathbb S}^{n}$. Thus, Theorems~\ref{th:csl} and \ref{th:csc} imply that $\nabla f$ is Lipschitz continuous with constant $2\beta$, and $f$ is strongly convex with constant $2\beta -\gamma$, respectively. Furthermore, as $C \subset {\mathbb S}^{n}_{+}$, we conclude that $C_\infty \subset {\mathbb S}^{n}_{+}$. Hence, ${\mathbb S}^{n}_{+}={{\mathbb S}^{n}_{+}}^*\subset C_\infty^*$, which implies that ${\mathbb S}^{n}_{++} \subset int(C_\infty^*)$. As $\nabla f(x)\in {\mathbb S}^{n}_{++}$ for all $x\in C$, we conclude that $\nabla f(x) \in int(C_\infty^*)$ for all $x\in C$. Moreover the convexity of $g$ implies that $C$ is convex. Therefore, the problem~\eqref{eq:mp} is convex, and $f$ satisfies conditions {\bf (A)} and {\bf (B)}. In particular, if $g: {\mathbb S}^{n}_{++}\to \R$ given by $g(x)=1/det(X)$ or $g(x)=1/tr(X)$, then the set $C\subset {\mathbb S}^{n}_{++}$ is convex and unbounded.
\end{example}
\section{Conclusions} \label{Sec:Conclusions}
In this study, we considered the classical Frank--Wolf algorithm for nonempty, closed, convex, not necessarily compact constraints. To study its convergence properties, we used recession techniques. The examples in Section 4 demonstrated that the Frank--Wolf algorithm can indeed be applied to several optimization problems (not necessarily convex) with non-compact constraint sets.
\section*{Funding}
The first author was supported in part by CNPq grants 305158/2014-7 and 302473/2017-3, FAPEG/PRONEM- 201710267000532, and CAPES. The second author was supported in part by Funda\c{c}\~{a}o de Apoio \`{a} Pesquisa do Distrito Federal (FAP-DF) through grants 0193.001695/2017 and PDE 05/2018. This research was partly carried out during a visit of the second author to the Center for Mathematical Research (CRM) (while Western Catalonia was in a state of alert), in the framework of the Research-in-pairs call in 2020. The CRM is a paradise for research, and the second author appreciates its hospitality and support.

\end{document}